\documentclass[10pt]{amsart}
\usepackage[utf8]{inputenc}
\usepackage[margin=1.0in]{geometry}
\usepackage{amssymb,amsmath,amsthm}
\usepackage{amsfonts}
\usepackage{amsmath}
\usepackage{xcolor}
\usepackage{amsthm}
\usepackage{pdflscape}
\usepackage{pgfplots}
\usepackage{mathrsfs}
\usepackage{mathbbol}
\usepackage{multirow}
\usepackage[numbers]{natbib}
\usepackage{enumerate}
\usepackage{enumitem}
\usepackage{hyperref}
\usepackage{tikz-cd}
\usepackage{etoolbox}
\usepackage{setspace}
\hypersetup{
pdftitle={On the quantum differentiation of smooth real-valued functions},
pdfsubject={},
pdfauthor={Kolosov Petro},
pdfkeywords={}
}
\usepackage[symbol]{footmisc}
\usepackage{colortbl}
\usepackage{xcolor}
\definecolor{shadegray}{RGB}{128, 128, 128}
\definecolor{lightgray}{RGB}{200,200,200}
\usepackage{amsaddr}

\theoremstyle{definition}
\newtheorem{thm}{Theorem}[section]

\newtheorem{prop}[thm]{Proposition}
\newtheorem{lem}[thm]{Lemma}

\newtheorem{quest}[thm]{Question}

\newtheorem*{proposition*}{Proposition}

\makeatletter
\let\c@equation\c@thm
\makeatother
\numberwithin{equation}{section}

\date{}
\title{Invariants of Unimodular Quadratic Polynomial Poisson Algebras of Dimension 3}
\author{CHENGYUAN MA}
\address{Department of Mathematics, University of Washington}
\email{c9ma@uw.edu}

\let\oldtocsection=\tocsection
\let\oldtocsubsection=\tocsubsection
\let\oldtocsubsubsection=\tocsubsubsection 
\renewcommand{\tocsection}[2]{\hspace{0em}\oldtocsection{#1}{#2}}
\renewcommand{\tocsubsection}[2]{\hspace{1em}\oldtocsubsection{#1}{#2}}
\renewcommand{\tocsubsubsection}[2]{\hspace{2em}\oldtocsubsubsection{#1}{#2}}
\pgfplotsset{compat=1.18}

\begin{document}
\subjclass[2020]{17B63, 16R30}
\keywords{Poisson alegbra $\cdot$ Deformation Quantization $\cdot$ Invariant Subalgebra $\cdot$ Reflection $\cdot$ Rigidity}

\maketitle

\pagestyle{plain}

\setcounter{section}{-1}

\setstretch{1.2}

\begin{abstract}
Let $P = \Bbbk[x_1, x_2, x_3]$ be a unimodular quadratic Poisson algebra, with its Poisson bracket written as $\{x_i, x_j\} = \displaystyle{\sum_{k,l}c_{i,j}^{k,l}x_kx_l}$, $1 \leq i < j \leq 3$. Let $P_{\hbar}$ be the deformation quantization of $P$ constructed as follows: $P_{\hbar} = \Bbbk\langle y_1, y_2, y_3\rangle/([y_i,y_j]=\frac{\hbar}{2}\displaystyle{\sum_{k,l}}c_{i,j}^{k,l}(y_ky_l+y_ly_k))_{1 \leq i < j \leq 3}$. In this paper, we establish that $P$ and $P_{\hbar}$ possess identical graded automorphisms and reflections, and that taking invariant subalgebras and taking deformation quantizations are two commutative processes.
\end{abstract}

\addtocontents{toc}{\protect\setcounter{tocdepth}{0}}
\section{Introduction}
\addtocontents{toc}{\protect\setcounter{tocdepth}{1}}

Throughout $\Bbbk$ is an algebraically closed base field of characteristic 0. The classical Shephard-Todd-Chevalley Theorem states that if $G \subseteq \text{Aut}_{\text{gr}}(\Bbbk[x_1, \cdots, x_n])$ is a finite subgroup, then the invariant subalgebra $\Bbbk[x_1, \cdots, x_n]^{G}$ is regular if and only if $G$ is generated by (pseudo-)reflections \cite{ST} \cite{C}. Since the 70s, scholars have strived to recover a variant of the Shephard-Todd-Chevalley Theorem in non-commutative settings. Prominent examples include, but are certainly not limited to, universal enveloping algebra of semisimple Lie algebras and Weyl algebras in \cite{AP}, skew polynomial rings and quantum matrix algebras in \cite{KKZ2}, non-PI Sklyanin algebras of global dimension $\geq 3$ in \cite{KKZ}, down-up algebras in \cite{KKZ3}, semiclassical limits of specific Artin-Schelter regular algebras in \cite{GVW}. In this paper, our principal focus lies in extending \cite{Ma1}'s investigation of the Shephard-Todd-Chevalley Theorem concering unimodular quadratic Poisson structures on $\Bbbk[x_1, x_2, x_3]$. Our objective is to establish an analogous version of the theorem tailored to the standard deformation quantizations of the Poisson algebras studied by \cite{Ma1}. Prior to stating the main theorems of this paper, we shall begin by revisiting several fundamental concepts.

\smallskip

A $\Bbbk$-algebra $A$ is called \textit{Artin-Schelter regular} if 
\begin{enumerate}[label = (\arabic*)]
    \item $A$ is connected $\mathbb{N}$-graded: $A$ admits a $\Bbbk$-vector space decomposition $A = \displaystyle{\bigoplus_{n \in \mathbb{N}}A_n}$ such that $A_0 = \Bbbk$ and $A_{i}A_{j} \subseteq A_{i+j}$ for all $i, j \in \mathbb{N}$.

    \item $A$ has finite Gelfand-Kirillov dimension: $\dim_{\Bbbk}A_n$ has polynomial growth. 

    \item $A$ has finite global dimension $d$.

    \item $A$ is Gorenstein: $\text{Ext}_{A}^{i}(\Bbbk, A) \cong \begin{cases}
        0 & i \neq d\\
        \Bbbk(l) & i = d
    \end{cases}$, for some $l \in \mathbb{Z}$. 
\end{enumerate}

Let $A$ be an Artin-Schelter regular algebra and let $\phi$ be a graded automorphism of $A$. The \textit{trace series} of $\phi$ is $\text{Tr}_{A}(\phi,t) = \displaystyle{\sum_{i = 0}^{\infty}\text{tr}(\phi\big\vert_{A_i})t^i}$. If further $A$ has Hilbert series $h_{A}(t) = \displaystyle{\frac{1}{(1-t)^{n}f(t)}}$ for some $f(t)$ satisfying $f(1) \neq 0$, a \textit{quasi-reflection} of $A$ is a finite-order graded automorphism $\phi: A \to A$ such that $\text{Tr}_{A}(\phi,t) = \displaystyle{\frac{1}{(1-t)^{n-1}g(t)}}$ for some $g(t)$ satisfying $g(1) \neq 0$. We will write $\text{QR}(A)$ to denote the set of quasi-reflections of $A$. 

\smallskip

Artin-Schelter regular algebras, initially introduced in \cite{AS}, is an important class of non-commutative algebras possessing the ``appropriate" properties for transferring results from the realm of commutative algebras to the realm of non-commutative algebras. Artin-Schelter regular algebras are commonly regarded as the non-commutative counterparts of polynomial algebras because they possess growth conditions and homological conditions that are inherent to polynomial algebras. 

\smallskip

A \textit{Poisson algebra} is a commutative $\Bbbk$-algebra $P$ together with a Lie bracket $\{-,-\}: P^{\otimes 2} \to P$ satisfying the Leibniz rule: $\{a,bc\} = b\{a,c\} + \{a,b\}c$ for all $a, b, c \in P$. Henceforth, we will assume that $P = \Bbbk[x_1, \cdots, x_n]$ is a connected $\mathbb{N}$-graded Poisson algebra under the standard grading and we will write $\text{PAut}_{\text{gr}}(P)$ to denote its graded Poisson automorphism group. A \textit{Poisson reflection} of $P$ is a finite-order graded Poisson automorphism $\phi$ such that $\phi\big\vert_{P_1}$ has eigenvalues $\underbrace{1,\cdots,1}_{n-1}, \xi$, for some primitive $m$th root of unity $\xi$. We will write $\text{PR}(P)$ to denote the set of Poisson reflections of $P$. 

\smallskip

Poisson algebras originally emerged in classical mechanics and subsequently assumed a significant role in mathematical physics. In recent decades, Poisson algebras have also garnered attention in pure mathematics.One explanation for this heightened interest can be attributed to their approximity with Artin-Schelter regular algebras through deformation quantizations, semiclassical limits, and Poisson enveloping algebras. 
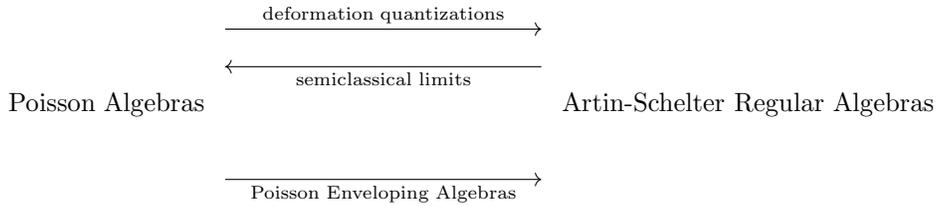
\[
\begin{tikzcd}
\text{Poisson Algebras }
\arrow[rrrrr, shift={(0,1)}, "\text{deformation quantizations}"] 
\arrow[rrrrr, shift={(0,-1)}, "\text{Poisson Enveloping Algebras}", swap]
& & & & & \text{ Artin-Schelter Regular Algebras}
\arrow[lllll, shift={(0,0.5)}, "\text{semiclassical limits}"]
\end{tikzcd}
\]

\smallskip

Let $A$ be a $\Bbbk[[\hbar]]$-algebra such that $A/\hbar A$ is a commutative algebra. Let $f, g \in A$, $(f + \hbar A)(g + \hbar A) = (g + \hbar A)(f + \hbar A)$ and therefore $[f, g] = fg - gf = \hbar\gamma(f, g)$ for some $\gamma(f, g) \in A$. Let $\overline{f}, \overline{g}$ be the projections of $f, g$ on $A/\hbar A$, respectively. Define $\{-,-\}: (A/\hbar A)^{\otimes 2} \to A/\hbar A$ as follows:
\[
\{\overline{f}, \overline{g}\} = \overline{\gamma(f,g)},
\]
Then $(A/\hbar A, \{-,-\})$ is a Poisson algebra, called the \textit{semiclassical limit} of $A$.

\smallskip

Let $P$ be a Poisson $\Bbbk$-algebra. A \textit{(graded) deformation quantization} of $P$ is a graded $\Bbbk[[\hbar]]$-algebra $P_{\hbar}$ satisfying the following conditions:
\begin{enumerate}[label = (\arabic*)]
    \item $P_{\hbar} \cong P[[\hbar]]$ as graded $\Bbbk[[\hbar]]$-modules.

    \item $P_{\hbar}/\hbar P_{\hbar} \cong P$ as graded $\Bbbk$-algebras. In particular, for all $f, g \in P_{\hbar}$, $(f + \hbar P_{\hbar})(g + \hbar P_{\hbar}) = (g + \hbar P_{\hbar})(f + \hbar P_{\hbar})$ and therefore $[f, g] = fg - gf = \hbar\gamma(f, g)$ for some $\gamma(f, g) \in P_{\hbar}$.

    \item Let $f, g \in P_{\hbar}$ and let $\overline{f}, \overline{g}$ be the projections of $f, g$ on $P_{\hbar}/\hbar P_{\hbar} \cong P$, respectively. Then $\{\overline{f}, \overline{g}\} = \overline{\gamma(f,g)}$.
\end{enumerate}

The Poisson algebra $P$ is called \textit{quantizable} if $P$ admits a deformation quantization $P_{\hbar}$. It is possible for a Poisson algebra to possess multiple non-isomorphic deformation quantizations; nevertheless, within the scope of this thesis, our primary focus lies on a variant of the deformation quantizations constructed by \cite{DML}.

\smallskip

Let $P = \Bbbk[x_1, x_2, x_3]$ be a quadratic Poisson algebra, with its Poisson structure written as $\{x_i, x_j\} = \displaystyle{\sum_{k,l}c_{i,j}^{k,l}x_kx_l}$, $1 \leq i < j \leq 3$. Donin and Makar-Limanov proved that $P$ is quantizable, admitting the following deformation quantization \cite{DML}:
\[
P_{\hbar} = \Bbbk\langle y_1, y_2, y_3\rangle/([y_i,y_j]=\frac{\hbar}{2}\sum_{k,l}c_{i,j}^{k,l}(y_ky_l+y_ly_k))_{1 \leq i,j \leq 3},
\]
for some $0 \neq \hbar \in \Bbbk$. It is possible to prove that $P_{\hbar}$ is a quantum polynomial ring and $\text{PAut}_{\text{gr}}(P) \cong \text{Aut}_{\text{gr}}(P_{\hbar})$. If in addition $P$ is unimodular, then $\text{PR}(P) = \text{QR}(P_{\hbar})$.

\smallskip

The main theorems of this paper can be stated as follows:

\begin{thm}
Let $P = \Bbbk[x_1, x_2, x_3]$ be a unimodular quadratic Poisson algebra. Let $G$ be a finite subgroup of the graded Poisson automorphism group of $P$, and $G_{\hbar}$ be the corresponding finite subgroup of the graded automorphism group of $P_{\hbar}$ under the isomorphism $\text{PAut}_{\text{gr}}(P) \cong \text{Aut}_{\text{gr}}(P_{\hbar})$. The following are equivalent:
\begin{enumerate}[label = (\arabic*)]
    \item $G$ is generated by Poisson reflections.

    \item $G_{\hbar}$ is generated by quasi-reflections.

    \item $P_{\hbar}^{G_{\hbar}}$ is Artin-Schelter regular.
\end{enumerate}
\end{thm}

\begin{thm}
Let $P = \Bbbk[x_1, x_2, x_3]$ be a unimodular quadratic Poisson algebra. Let $G$ be a finite subgroup of the graded Poisson automorphism group of $P$ and let $G_{\hbar}$ be the corresponding finite subgroup of the graded automorphism group of $P_{\hbar}$ under the isomorphism $\text{PAut}_{\text{gr}}(P) \cong \text{Aut}_{\text{gr}}(P_{\hbar})$. Define $Q_{\hbar} := P_{\hbar}^{G_{\hbar}}$, with $\hbar$ viewed as a formal parameter (as opposed to a scalar value). Then 
\begin{enumerate}[label = (\arabic*)]
    \item $Q_{\hbar}/(\hbar) \cong P^{G}$ as $\Bbbk$-algebras.

    \item For all $f, g \in Q_{\hbar}$, $fg - gf = \hbar \pi_1(f,g)$ for some $\pi_1(f,g) \in Q_{\hbar}$.

    \item $Q_{\hbar}/(\hbar)$ together with the following Poisson bracket:
    \[
    \{\overline{f}, \overline{g}\} = \overline{\pi_1(f,g)},
    \]
    where $\overline{(\hspace{.1in})}$ denotes the image under the natural projection $Q_{\hbar} \to Q_{\hbar}/(\hbar)$, is isomorphic to $P^G$ as Poisson algebras. 
\end{enumerate}

It is worth noting that Theorem 0.1 is a variation of the Shephard-Todd-Chevalley Theorem that establishes a connection between unimodular quadratic Poisson structures on $\Bbbk[x_1, x_2, x_3]$ and their deformation quantizations in the sense of \cite{DML}, and Theorem 0.2 asserts that taking invariant subalgebras under finite reflection groups and taking deformation quantizations are two commutative processes:
\[
\begin{tikzcd}
\rotatebox{0}{$P$} \arrow{rrrrr}{\rotatebox{0}{deformation quantization}} \arrow[swap]{ddd}{\hspace*{-0.5cm}\rotatebox{90}{invariant}} & & & & & \rotatebox{0}{$P_{\hbar}$}
\arrow{ddd}{\hspace*{0.1cm}\rotatebox{270}{invariant}}\\
\\
\\
\rotatebox{0}{$P^{G}$} & & & & & \rotatebox{0}{$P_{\hbar}^{G_{\hbar}}$} 
\arrow{lllll}{\hspace*{0.2cm}\rotatebox{0}{semiclassical limit}}
\end{tikzcd}
\]
\end{thm}

\smallskip

The structure of this paper is as follows: in Section 1, we review concepts that are relevant to this paper but have not been covered in the introduction; in Section 2 and Section 3, we establish a parallel between graded Poisson automorphism groups (resp. Poisson reflections) of unimodular quadratic Poisson structures on $\Bbbk[x_1, x_2, x_3]$ and their deformation quantizations in the sense of \cite{DML}; in Section 4, we prove Theorem 0.1; in Section 5, we prove Theorem 0.2; finally in Section 6, we give several remarks and a list of possible future research questions.

\

\section{Preliminaries}

Let $P = \Bbbk[x_1, x_2, x_3]$ be a Poisson algebra under the standard grading. 
\begin{itemize}
    \item The Poisson algebra $P$ is called \textit{quadratic} if $\{P_1, P_1\} \subseteq P_2$.  

    \item The Poisson algebra $P$ is called \textit{unimodular} if the Poisson structure $\{-,-\}$ has the following form:
    \[
    \{x_1, x_2\} = \frac{\partial \Omega}{\partial x_3}, \quad
    \{x_2, x_3\} = \frac{\partial \Omega}{\partial x_1}, \quad
    \{x_3, x_1\} = \frac{\partial \Omega}{\partial x_2},
    \]
for some homogeneous polynomial $\Omega$ of $x_1$, $x_2$, $x_3$.
\end{itemize}

Throughout the paper, we are primarily interested in the Poisson structures on $P$ that are both unimodular and quadratic. The classification of such Poisson structures, or equivalently, the classification of superpotential $\Omega$ can be located in a number of papers, for instrance \cite{BW}.
\begin{center}
\renewcommand{\arraystretch}{1.25}
\begin{tabular}{ |p{1cm}||p{3cm}||p{3.5cm}|p{3.5cm}|p{3.5cm}| } 
\hline
\centering \textbf{Case} & \centering $\boldsymbol{\Omega}$ & \centering $\boldsymbol{\{x_1,x_2\}}$ & \centering $\boldsymbol{\{x_2,x_3\}}$ & \centering $\boldsymbol{\{x_3,x_1\}}$ \tabularnewline
\hline
\hline
\centering 1 & \centering $x_1^3$ & \centering 0 & \centering $3x_1^2$ & \centering 0 \tabularnewline
\hline
\centering 2 & \centering $x_1^2x_2$ & \centering 0 & \centering $2x_1x_2$ & \centering $x_1^2$ \tabularnewline
\hline
\centering 3 & \centering $2x_1x_2x_3$ & \centering $2x_1x_2$ & \centering $2x_2x_3$ & \centering $2x_1x_3$\tabularnewline
\hline
\centering 4 & \centering $x_1^2x_2+x_1x_2^2$ & \centering $0$ & \centering $2x_1x_2+x_2^2$ & \centering $x_1^2+2x_1x_2$ \tabularnewline
\hline
\centering 5 & \centering $x_1^3 + x_2^2x_3$ & \centering $x_2^2$ & \centering $3x_1^2$ & \centering $2x_2x_3$ \tabularnewline
\hline
\centering 6 & \centering $x_1^3 + x_1^2x_3 + x_2^2x_3$ & \centering $x_1^2+x_2^2$ & \centering $3x_1^2+2x_1x_3$ & \centering $2x_2x_3$ \tabularnewline
\hline
\centering 7 & \centering $\frac{1}{3}(x_1^3+x_2^3+x_3^3)+\lambda x_1x_2x_3, \lambda^3 \neq -1$ & \centering $x_3^2+\lambda x_1x_2$ & \centering $x_1^2+\lambda x_2x_3$ & \centering $x_2^2+\lambda x_1x_3$ \tabularnewline
\hline
\centering 8 & \centering $x_1^3+x_1^2x_2+x_1x_2x_3$ & \centering $x_1x_2$ & \centering $3x_1^2+2x_1x_2+x_2x_3$ & \centering $x_1^2+x_1x_3$ \tabularnewline
\hline
\centering 9 & \centering $x_1^2x_3 + x_1x_2^2$ & \centering $x_1^2$ & \centering $2x_1x_3+x_2^2$ & \centering $2x_1x_2$ \tabularnewline
\hline
\end{tabular}
\end{center}

\medskip

Let $P = \Bbbk[x_1, x_2, x_3]$ be a quadratic Poisson algebra, with its Poisson structure written as $\{x_i, x_j\} = \displaystyle{\sum_{k,l}c_{i,j}^{k,l}x_kx_l}$, $1 \leq i < j \leq 3$. Donin and Makar-Limanov proved that $P$ is quantizable, admitting the following deformation quantization:
\[
P_{\hbar} = \Bbbk[[\hbar]]\langle y_1, y_2, y_3\rangle/([y_i,y_j]=\frac{\hbar}{2}\sum_{k,l}c_{i,j}^{k,l}(y_ky_l+y_ly_k))_{1 \leq i,j \leq 3.}
\]
In practice, we will work with a variant of the above deformation quantizations, modified by replacing the coefficient ring $\Bbbk[[\hbar]]$ with an algebraically closed field $\mathbb{F}$ of characteristic 0. In a more formal manner, we can replace $P_{\hbar}$ with $P_{\hbar} \otimes_{\Bbbk[[\hbar]]} \overline{\Bbbk((\hbar))}$, or equivalently, we can replace the coefficient ring $\Bbbk[[\hbar]]$ by $\mathbb{F} = \overline{\Bbbk((\hbar))}$, the Puiseux series of $\hbar$ over $\Bbbk$. Given that our attention rarely extends to scalars, we will assume that the formal parameter $\hbar$ has been replaced by a non-zero scalar in $\Bbbk$ and that $\mathbb{F} = \Bbbk$. 

\begin{thm}
\cite[Theorem 3.1]{DML} Let $P = \Bbbk[x_1, x_2, x_3]$ be a quadratic Poisson algebra, with its Poisson structure written as $\{x_i, x_j\} = \displaystyle{\sum_{k,l}c_{i,j}^{k,l}x_kx_l}$, $1 \leq i < j \leq 3$. Then the Poisson algebra $P$ is quantizable, admitting the standard deformation quantization:
\[
P_{\hbar} = \Bbbk\langle y_1, y_2, y_3\rangle/([y_i,y_j]=\frac{\hbar}{2}\sum_{k,l}c_{i,j}^{k,l}(y_ky_l+y_ly_k))_{1 \leq i,j \leq 3},
\]
for some $\hbar \in \Bbbk^{\times}$. 
\end{thm}

Working with the standard deformation quantization offers two benefits:
\begin{enumerate}[label = (\arabic*)]
    \item It is imperative to set $A_0 = \Bbbk$, as Artin-Schelter regularity necessitates a connected $\mathbb{N}$-grading. In the later part of this section, we will prove that the deformation quantization in Theorem 1.1 satisfies Artin-Schelter regularity. 

    \item It is computationally (and psychologically) easier to work with scalars from a field rather than with scalars from a formal power series ring.
\end{enumerate}

For the remainder of the paper, we will refer this deformation quantization as the \textit{standard deformation quantization} of $P$. The standard deformation quantization of a unimodular quadratic Poisson structures on $\Bbbk[x_1, x_2, x_3]$ is a free $\Bbbk$-alegrba $\Bbbk\langle y_1, y_2, y_3 \rangle$ subjecting to the following relations:
\begin{center}
\renewcommand{\arraystretch}{1.25}
\begin{tabular}{ |p{1.5cm}||p{4.25cm}|p{4.25cm}|p{4.25cm}| } 
\hline
\centering \textbf{Case} & \centering $\boldsymbol{y_2y_1=}$ & \centering $\boldsymbol{y_3y_2=}$ & \centering $\boldsymbol{y_3y_1=}$ \tabularnewline
\hline
\hline
\centering 1 & \centering $y_1y_2$ & \centering $y_2y_3 - 3\hbar y_1^2$ & \centering $y_1y_3$ \tabularnewline
\hline
\centering 2 & \centering $y_1y_2$ & \centering $y_2y_3 - 2\hbar y_1y_2$ & \centering $y_1y_3 + \hbar y_1^2$ \tabularnewline
\hline
\centering 3 & \centering $\frac{1-\hbar}{1+\hbar}y_1y_2$ & \centering $\frac{1-\hbar}{1+\hbar}y_2y_3$ & \centering $\frac{1+\hbar}{1-\hbar}y_1y_3$\tabularnewline
\hline
\centering 4 & \centering $y_1y_2$ & \centering $y_2y_3 - 2\hbar y_1y_2 - \hbar y_2^2$ & \centering $y_1y_3 + \hbar y_1^2 + 2\hbar y_1y_2$ \tabularnewline
\hline
\centering 5 & \centering $y_1y_2 - \hbar y_2^2$ & \centering $y_2y_3 - 3\hbar y_1^2$ & \centering $(y_1 + 2\hbar y_2)y_3 -3\hbar^2 y_1^2$ \tabularnewline
\hline
\centering 6 & \centering $y_1y_2 - \hbar y_1^2 - \hbar y_2^2$ & \centering $(-\frac{2\hbar}{1+\hbar^2}y_1 + \frac{1-\hbar^2}{1+\hbar^2}y_2)y_3-\frac{3\hbar}{1+\hbar^2}y_1^2$ & \centering $(\frac{1-\hbar^2}{1+\hbar^2}y_1+\frac{2\hbar}{1+\hbar^2}y_2)y_3-\frac{3\hbar^2}{1+\hbar^2}y_1^2$ \tabularnewline
\hline
\centering 7 & \centering $\frac{2-\lambda\hbar}{2+\lambda\hbar}y_1y_2-\frac{2\hbar}{2+\lambda\hbar}y_3^2$ & \centering $\frac{2-\lambda\hbar}{2+\lambda\hbar}y_2y_3-\frac{2\hbar}{2+\lambda\hbar}y_1^2$ & \centering $\frac{2+\lambda\hbar}{2-\lambda\hbar}y_1y_3+\frac{\hbar}{2-\lambda\hbar}y_2^2$ \tabularnewline
\hline
\centering 8 & \centering $\frac{2-\hbar}{2+\hbar}y_1y_2$ & \centering $\frac{2-\hbar}{2+\hbar}y_2y_3 - \frac{6\hbar}{2+\hbar}y_1^2-\frac{8\hbar}{(2+\hbar)^2}y_1y_2$ & \centering $\frac{2+\hbar}{2-\hbar}y_1y_3 + \frac{2\hbar}{2-\hbar}y_1^2$ \tabularnewline
\hline
\centering 9 & \centering $y_1y_2-\hbar y_1^2$ & \centering $(-2\hbar y_1 + y_2)y_3+\hbar^3y_1^2-2\hbar^2y_1y_2-\hbar y_2^2$ & \centering $y_1y_3 - \hbar^2y_1^2+2\hbar y_1y_2$ \tabularnewline
\hline
\end{tabular}
\end{center}

\begin{lem}
Let $P = \Bbbk[x_1, x_2, x_3]$ be a quadratic Poisson algebra. Then its standard deformation quantization $P_{\hbar}$ is a Noetherian Artin-Schelter regular domain of global dimension 3 with Hilbert series $h_{P_{\hbar}}(t) = \displaystyle{\frac{1}{(1-t)^3}}$. In other words, its standard deformation quantization is a quantum polynomial ring of global dimension 3.
\end{lem}
\begin{proof}
In Case 1 to 6, 8 and 9, the standard deformation quantization is an Ore extension of either the skew polynomial ring of two variables or the Jordan plane, and therefore takes the form of an iterated Ore extension $\Bbbk[y_1][y_2; \sigma_2, \delta_2][y_3; \sigma_3, \delta_3]$, where appropriate selections of graded automorphisms $\sigma_2$ and $\sigma_3$, along with $\sigma_2$-derivation $\delta_2$ and $\sigma_3$-derivation $\delta_3$, are made. This can be readily verified by employing the following fact: the graded automorphisms of the skew polynomial ring of two variables is $\left\{\begin{bmatrix}a & 0\\0 & b\end{bmatrix}: a, b \neq 0\right\}$ when the relation is written as $y_2y_1 - \lambda y_1y_2$ ($\lambda \neq 0, \pm 1$) and $\left\{\begin{bmatrix}a & b\\-b & a\end{bmatrix}: a, b \neq 0\right\}$ when the relation is written as $y_2y_1 - y_1y_2 + \lambda y_1^2 + \lambda y_2^2$ ($\lambda \neq 0, \pm 1$), and the graded automorphisms of the Jordan plane is $\left\{\begin{bmatrix}a & 0\\b & a\end{bmatrix}: a \neq 0\right\}$ when the relation is written as $y_2y_1 - y_1y_2 - \lambda y_1^2$ ($\lambda \neq 0$). These iterated Ore extensions are Noetherian Artin-Schelter regular domains, as shown in \cite[Proposition 2]{AST} and \cite[Corollary 2.7, Exercise 2O]{GW}. For the Hilbert series, the standard deformation quantization $P_{\hbar}$ admits a $\Bbbk$-linear basis $\{y_1^iy_2^jy_3^k: i, j, k \geq 0\}$ \cite[page 254]{DML}. Consequently, the Hilbert series $h_{P_{\hbar}}(t) = \displaystyle{\frac{1}{(1-t)^3}}$.

\smallskip

In Case 14-7, the standard deformation quantization is the three-dimensional Sklyanin algebra, a family of Artin-Schelter regular algebras initially studied in \cite{AS} and \cite{ATV}. The proof of the remaining properties: Noetherian, domain, and Hilbert series, can be located in a number of references, for example \cite[Theorem 2.14]{W2}.
\end{proof}

\

\section{Graded Automorphisms of Standard Deformation Quantizations}

In this section, our objective is to provide a classification of graded automorphisms for the standard deformation quantizations of quadratic Poisson structures on $\Bbbk[x_1, x_2, x_3]$.

\smallskip

Let $P = \Bbbk[x_1, x_2, x_3]$ be a quadratic Poisson algebra and $P_{\hbar}$ be its standard deformation quantization. Suppose that $\phi \in \text{Aut}_{\text{gr}}(P_{\hbar})$. The graded automorphism $\phi$ can be uniquely represented by a $3 \times 3$ invertible matrix $\begin{bmatrix}
a_{11} & a_{12} & a_{13}\\a_{21} & a_{22} & a_{23}\\ a_{31} & a_{32} & a_{33}
\end{bmatrix}$ for some $a_{ij} \in \Bbbk$, $1 \leq i, j \leq 3$, satisfying $\displaystyle{\phi(y_i) = \sum_{j=1}^{3}a_{ij}y_j}$, $1 \leq i \leq 3$ and the following compatibility conditions on $a_{ij}$'s:
\[
\frac{\hbar}{2}\sum_{k,l}c_{i,j}^{k,l}\phi(y_ky_l+y_ly_k) = \phi([y_i,y_j]), 1 \leq i, j \leq 3.
\]

Expand the RHS: 

\begin{align*}
\frac{\hbar}{2}\sum_{k,l}c_{i,j}^{k,l}\phi(y_ky_l+y_ly_k) =&
(a_{i1}a_{j2}-a_{i2}a_{j1})[y_1,y_2]\\
+&(a_{i2}a_{j3}-a_{i3}a_{j2})[y_2,y_3]\\
+&(a_{i3}a_{j1}-a_{i1}a_{j3})[y_3,y_1]\\
=& (a_{i1}a_{j2}-a_{i2}a_{j1})\frac{\hbar}{2}\sum_{k',l'}c_{1,2}^{k',l'}(y_{k'}y_{l'}+y_{l'}y_{k'})\\
+&(a_{i2}a_{j3}-a_{i3}a_{j2})\frac{\hbar}{2}\sum_{k',l'}c_{2,3}^{k',l'}(y_{k'}y_{l'}+y_{l'}y_{k'})\\
+&(a_{i3}a_{j1}-a_{i1}a_{j3})\frac{\hbar}{2}\sum_{k',l'}c_{3,1}^{k',l'}(y_{k'}y_{l'}+y_{l'}y_{k'}).
\end{align*}

Given that the coefficients of $y_ky_l$ and $y_ly_k$ on the LHS coincide, and that the coefficients of $y_{k'}y_{l'}$ and $y_{l'}y_{k'}$ on the RHS also coincide, we may reduce the equation to:
\begin{align*}
\sum_{k,l}c_{i,j}^{k,l}\phi(y_ky_l)
=& (a_{i1}a_{j2}-a_{i2}a_{j1})\sum_{k',l'}c_{1,2}^{k',l'}y_{k'}y_{l'}\\
+&(a_{i2}a_{j3}-a_{i3}a_{j2})\sum_{k',l'}c_{2,3}^{k',l'}y_{k'}y_{l'}\\
+&(a_{i3}a_{j1}-a_{i1}a_{j3})\sum_{k',l'}c_{3,1}^{k',l'}y_{k'}y_{l'}.
\end{align*}

In the following lemma, we state the equivalency of classifying the graded automorphisms of $P_{\hbar}$ and classifying the graded Poisson automorphisms of $P$.
\begin{lem}
Let $P = \Bbbk[x_1, x_2, x_3]$ be a quadratic Poisson algebra and let $P_{\hbar}$ be its standard deformation quantization. Then $\text{Aut}_{\text{gr}}(P_{\hbar}) \cong \text{PAut}_{\text{gr}}(P)$ as groups.
\end{lem}
\begin{proof}
Let $\phi_{\hbar} \in \text{Aut}_{\text{gr}}(P_{\hbar})$, written as $\begin{bmatrix}
a_{11} & a_{12} & a_{13}\\a_{21} & a_{22} & a_{23}\\ a_{31} & a_{32} & a_{33}
\end{bmatrix}$ for some $a_{ij} \in \Bbbk$, $1 \leq i, j \leq 3$. Define $\phi: P \to P$ as follows: $\phi(x_i) = \displaystyle{\sum_{j=1}^{3}a_{ij}x_j}$, $1 \leq i \leq 3$. It is clear that invertibility of $\phi_{\hbar}$ is equivalent to the invertibility of $\phi$ and it remains to show that $\phi$ preserves the Poisson bracket. For $1 \leq i < j \leq 3$, 
\begin{align*}
\phi(\{x_i,x_j\}) =& (a_{i1}a_{j2}-a_{i2}a_{j1})\{x_1,x_2\}
+ (a_{i2}a_{j3}-a_{i3}a_{j2})\{x_2,x_3\}
+ (a_{i3}a_{j1}-a_{i1}a_{j3})\{x_3,x_1\}.
\end{align*}

Since $\{x_i, x_j\} = \displaystyle{\sum_{k,l}c_{i,j}^{k,l}x_kx_l}$, the LHS can be expanded into:

\begin{align*}
\displaystyle{\frac{1}{2}\sum_{k',l'}c_{i,j}^{k',l'}\phi(x_{k'}x_{l'})}
=& (a_{i1}a_{j2}-a_{i2}a_{j1})\{x_1,x_2\}\\
+& (a_{i2}a_{j3}-a_{i3}a_{j2})\{x_2,x_3\}\\
+& (a_{i3}a_{j1}-a_{i1}a_{j3})\{x_3,x_1\}\\
=& (a_{i1}a_{j2}-a_{i2}a_{j1})\displaystyle{\sum_{k',l'}c_{1,2}^{k',l'}x_{k'}x_{l'}}\\
+& (a_{i2}a_{j3}-a_{i3}a_{j2})\displaystyle{\sum_{k',l'}c_{2,3}^{k',l'}x_{k'}x_{l'}}\\
+& (a_{i3}a_{j1}-a_{i1}a_{j3})\displaystyle{\sum_{k',l'}c_{3,1}^{k',l'}x_{k'}x_{l'}}.
\end{align*}

Notice that this condition is precisely what makes $\phi_{\hbar}$ a graded homomorphism of $P_{\hbar}$ (by a factor of $\displaystyle{\frac{\hbar}{2}}$). Up to this point, we have defined a mapping $\text{Aut}_{\text{gr}}(P_{\hbar}) \to \text{PAut}_{\text{gr}}(P): \phi_{\hbar} \mapsto \phi$ that is apparently a group homomorphism. Its inverse can be constructed in a similar manner. 
\end{proof}

Lemma 2.1 states that the classification of $\text{Aut}_{\text{gr}}(P_{\hbar})$ is equivalent to the classification of $\text{PAut}_{\text{gr}}(P)$. In practices, the latter is preferred due to its simplicity (smaller basis), and a comprehensive classification has already been provided in \cite[Proposition 2.2]{Ma1}.

\begin{prop}
\cite[Proposition 2.2]{Ma1} Let $P = \Bbbk[x_1, x_2, x_3]$ be one of the quadratic Poisson algebras and let $P_{\hbar}$ be its standard deformation quantization. Then $\text{PAut}_{\text{gr}}(P) \cong \text{Aut}_{\text{gr}}(P_{\hbar})$ contains the following automorphisms:
\end{prop}
\begin{center}
\renewcommand{\arraystretch}{1.25}
\begin{tabular}{ |p{1.5cm}||p{12.75cm}| } 
\hline
\centering \textbf{Case} & \centering $\boldsymbol{ \textbf{PAut}_{\textbf{gr}}(\textbf{P})} \boldsymbol{\cong } \textbf{Aut}_{\textbf{gr}}(\boldsymbol{P_{\hbar}})$
\tabularnewline
\hline
\hline
\centering 1 & \centering 
$\begin{bmatrix}
\pm \sqrt{bf-ce} & 0 & 0\\
a & b & c\\
d & e & f
\end{bmatrix}$ \tabularnewline
\hline
\centering 2 & \centering 
$\begin{bmatrix}
a & 0 & 0\\
0 & b & 0\\
c & d & a
\end{bmatrix}$
\tabularnewline
\hline
\centering 3 & \centering
$\begin{bmatrix}
a & 0 & 0\\
0 & b & 0\\
0 & 0 & c
\end{bmatrix}$, 
$\begin{bmatrix}
0 & a & 0\\
0 & 0 & b\\
c & 0 & 0
\end{bmatrix}$,
$\begin{bmatrix}
0 & 0 & a\\
b & 0 & 0\\
0 & c & 0
\end{bmatrix}$
\tabularnewline
\hline
\end{tabular}
\end{center}

\vspace{1.0cm}

\begin{center}
\renewcommand{\arraystretch}{1.25}
\begin{tabular}{ |p{1.5cm}||p{12.75cm}| } 
\hline
\centering \textbf{Case} & \centering $\boldsymbol{ \textbf{PAut}_{\textbf{gr}}(\textbf{P})} \boldsymbol{\cong } \textbf{Aut}_{\textbf{gr}}(\boldsymbol{P_{\hbar}})$
\tabularnewline
\hline
\hline
\centering 4 & \centering
$\begin{bmatrix}0 & a & 0\\-a & -a & 0\\b & c & a\end{bmatrix},
\begin{bmatrix}0 & -a & 0\\-a & 0 & 0\\b & c & a\end{bmatrix},
\begin{bmatrix}-a & 0 & 0\\a & a & 0\\b & c & a\end{bmatrix}$,\tabularnewline
& \centering $\begin{bmatrix}-a & -a & 0\\a & 0 & 0\\b & c & a\end{bmatrix},
\begin{bmatrix}a & a & 0\\0 & -a & 0\\b & c & a\end{bmatrix}, 
\begin{bmatrix}a & 0 & 0\\0 & a & 0\\b & c & a\end{bmatrix}$\tabularnewline
\hline
\centering 5 & \centering
$\begin{bmatrix}
a & 0 & 0\\
0 & a & 0\\
0 & 0 & a
\end{bmatrix}$
\tabularnewline
\hline
\centering 6 & \centering
$\begin{bmatrix}
a & 0 & 0\\
0 & a & 0\\
0 & 0 & a
\end{bmatrix}$,
$\begin{bmatrix}
-\frac{1}{2}a & \pm\frac{\sqrt{3}}{2}a & 0\\
\mp\frac{\sqrt{3}}{2}a & -\frac{1}{2}a & 0\\
\frac{9}{8}a & \mp\frac{3\sqrt{3}}{8}a & a
\end{bmatrix}$
\tabularnewline
\hline
\centering 7 & \centering
$\left(
\langle\begin{bmatrix}a & 0 & 0\\0 & a & 0\\0 & 0 & a\end{bmatrix}: a \neq 0\rangle \times \langle\begin{bmatrix}1 & 0 & 0\\0 & b & 0\\0 & 0 & b^2\end{bmatrix}: b = \xi_3, \xi_3^2\rangle
\right) \rtimes 
\langle\begin{bmatrix}0 & 1 & 0\\0 & 0 & 1\\1 & 0 & 0\end{bmatrix}\rangle$,
\tabularnewline
\hline
\centering 8 & \centering 
$\begin{bmatrix}
a & 0 & 0\\
0 & \frac{a^2}{b} & 0\\
b-a & 0 & b
\end{bmatrix}$
\tabularnewline
\hline
\centering 9 & \centering 
$\begin{bmatrix}
a & 0 & 0\\
b & a & 0\\
-\frac{b^2}{a} & -b & a
\end{bmatrix}$
\tabularnewline
\hline
\end{tabular}
\end{center}

\

\section{Quasi-Reflections of Deformation Quantizations}
In this section, our attention is centered on the classification of quasi-reflections for the standard deformation quantizations of unimodular Poisson structures on $\Bbbk[x_1, x_2, x_3]$. 

\smallskip

We begin by noting the following:
\begin{lem}
Let $P = \Bbbk[x_1, x_2, x_3]$ be a unimodular quadratic Poisson algebra and let $P_{\hbar}$ be its standard deformation quantization. Suppose that $\phi$ is a graded automorphism of $P_{\hbar}$. If $\phi$ satisfies the following conditions:
\begin{enumerate}[label=(\arabic*)]
    \item $\phi$ has the form $\phi\big\vert_{P_1} = 
    \begin{bmatrix}a_{11} & a_{12} & 0\\a_{21} & a_{22} & 0\\a_{31} & a_{32} & a_{33}\end{bmatrix}$,

    \item the coefficients of $y_1y_3$, $y_2y_3$, $y_3^2$ are equal to 0 in the commutator bracket $[y_1, y_2]$,

    \item the coefficient of $y_3^2$ are equal to 0 in the commutator brackets $[y_2, y_3]$ and $[y_3, y_1]$,
\end{enumerate}
and either of the following conditions:
\begin{enumerate}[label = (\arabic*)]    
    \item[(4)] the commutator bracket $[y_1, y_2]$ is equal to 0,

    \item[(4')] the entries $a_{12}$, $a_{21}$ are equal to 0,

    \item[(4'')] the entry $a_{12}$ is equal to 0 and the coefficient of $y_2^2$ in the commutator bracket $[y_1,y_2]$ is equal to 0, 
\end{enumerate}
then $\phi$ is a quasi-reflection if and only if $\phi$ is a classical reflection.
\end{lem}
\begin{proof}
Fix a degree $d \geq 0$. Take an arbitrary basis element $y_1^iy_2^jy_3^k$ from $(P_{\hbar})_d$: $i + j + k = d$ \cite[page 254]{DML}. Apply $\phi$ to $y_1^iy_2^jy_3^k$:
\[
\phi(y_1^iy_2^jy_3^k) = (a_{11}y_1 + a_{12}y_2)^i(a_{21}y_1 + a_{22}y_2)^j(a_{31}y_1 + a_{32}y_2 + a_{33}y_3)^k.
\]
To compute the trace series $\text{Tr}_{P_{\hbar}}(\phi, t)$, we determine the coefficient of the term $y_1^iy_2^jy_3^k$ after rearranging the terms of $\phi(y_1^iy_2^jy_3^k)$ in lexicographical order. 

Given that flipping $y_2$ and $y_1$ does not result in the generation of $y_3$, the exclusive source of $y_3$ would reside within the product $(a_{31}y_1 + a_{32}y_2 + a_{33}y_3)^k$. In order to generate $y_3^k$, the exclusive source is the product of $k$ copies of $a_{33}y_3$. This is due to the observation that choosing $a_{31}y_1$ or $a_{32}y_2$ in the product $(a_{31}y_1 + a_{32}y_2 + a_{33}y_3)^k$ results in a reduction by one in the number of copies of $y_3$, according to the constraints (2) and (3). 

It remains to generate $y_1^iy_2^j$. If either condition (4), (4'), or (4'') is met, it follows that the coefficient of the term $y_1^iy_2^j$ is identical to the coefficient of the term $y_1^iy_2^j$ when we rearrange the terms lexicographically subjecting to the relations $y_1y_2 = y_2y_1$. In conjunction with the above argument regarding $y_3^k$, the coefficient of the term $y_1^iy_2^jy_3^k$ is identical to the coefficient of the term $y_1^iy_2^jy_3^k$ when we rearrange the terms lexicographically in a polynomial ring, that is, subjecting to the relations $y_1y_2 = y_2y_1$, $y_2y_3 = y_3y_2$, $y_3y_1 = y_1y_3$. Consequently, 
\[
\text{Tr}_{P_{\hbar}}(\phi,t) = \frac{1}{\det(I_n - t\phi\big\vert_{(P_{\hbar})_1})} = \frac{1}{(1-\lambda_1t) \cdots (1 - \lambda_nt)},
\]
where $\lambda_1, \cdots, \lambda_n$ are the eigenvalues of the matrix $\phi\big\vert_{(P_{\hbar})_1}$. 

According to \cite[Theorem 3.1]{KKZ}, a quasi-reflection of a quantum polynomial ring, to which $P_{\hbar}$ belongs as established by Lemma 2.4.4, must be either a classical reflection or a mystic reflection. However, in the latter case, $\text{Tr}_{P_{\hbar}}(\phi, t) = \displaystyle{\frac{1}{(1-t)(1+t^2)} \neq \frac{1}{(1-t)^2f(t)}}$ for some polynomial $f(t)$ satisfying $f(1) \neq 0$. This rules out the possiblity of $\phi$ being a mystic reflection, leading to the conclusion that $\phi$ is a quasi-reflection if and only if $\phi$ is a classical reflection.
\end{proof}

Lemma 3.1 resolves 70\% of the instances of classifying quasi-reflections for unimodular quadratic Poisson structures on $\Bbbk[x_1, x_2, x_3]$. However, there are certain cases that remain unaddressed, and these will be discussed in the subsequent lemma.

\begin{lem}
Let $P = \Bbbk[x_1, x_2, x_3]$ be a unimodular quadratic Poisson algebra derived from the following superpotentials:
\[
\Omega_1: x_1^3, \quad \Omega_3: 2x_1x_2x_3, \quad \Omega_6: x_1^3+x_1^2x_3+x_2^2x_3, \quad \Omega_7: \frac{1}{3}(x_1^3+x_2^3+x_3^3)-\lambda x_1x_2x_3, \lambda^3 \neq 1.
\]
If $\phi$ is a graded automorphism of the standard deformation quantization $P_{\hbar}$ of $P$, then $\phi$ is a quasi-reflection if and only if $\phi$ is a classical reflection.
\end{lem}

\begin{proof}

For \textit{Unimodular 1} and \textit{Unimodular 3}, our objective is to establish that for any graded automorphism of the standard deformation quantization $P_{\hbar}$, the trace series 
\[
\displaystyle{\text{Tr}_{P_{\hbar}}(\phi, t) = \frac{1}{(1-\lambda_1t)\cdots(1-\lambda_n t)}},
\]
where $\lambda_1, \cdots, \lambda_n$ are the eigenvalues of $\phi\big\vert_{(P_{\hbar})_1}$. As elucidated in Lemma 4.3.1, this equality enables us to conclude that a classical reflection is necessarily a quasi-reflection, and vice versa. 

\textit{Unimodular 1}. $\Omega_1 = x_1^3$. 

A graded automorphism $\phi$ of $P_{\hbar}$ has the form $\phi\big\vert_{(P_{\hbar})_1} = \begin{bmatrix}\pm\sqrt{bf-ce} & 0 & 0\\a & b & c\\d & e & f\end{bmatrix}$. Fix a degree $d \geq 0$. Take an arbitrary basis element $y_1^iy_2^jy_3^k$ from $(P_{\hbar})_d$: $i + j + k = d$ \cite[page 254]{DML}. Apply $\phi$ to $y_1^iy_2^jy_3^k$:
\[
\phi(y_1^iy_2^jy_3^k) = \pm\sqrt{bf-ce}y_1^i(ay_1+by_2+cy_3)^j(dy_1+ey_2+fy_3)^k.
\]
To compute the trace series $\text{Tr}_{P_{\hbar}}(\phi, t)$, we determine the coefficient of the term $y_1^iy_2^jy_3^k$ after rearranging the terms of $\phi(y_1^iy_2^jy_3^k)$ in lexicographical order. 

As $\pm\sqrt{bf-ce}y_1^i$ already supplies $y_1^i$, and flipping the order of any of the variables will not result in an increase of $y_2$ and $y_3$, our attention is solely on $(by_2+cy_3)^j(ey_2+fy_3)^k$ within the last two multiplicands. Utilizing the commutator relation $[y_2, y_3] = 3\hbar y_1^2$, the coefficient of the term $y_1^iy_2^jy_3^k$ is identical to the coefficient of the term $y_1^iy_2^jy_3^k$ when we rearrange the terms lexicographically in a polynomial ring. Using the same reasoning as in Lemma 4.3.1, $\displaystyle{\text{Tr}_{P_{\hbar}}(\phi, t) = \frac{1}{(1-\lambda_1t)\cdots(1-\lambda_n t)}}$, where $\lambda_1, \cdots, \lambda_n$ are the eigenvalues of $\phi\big\vert_{(P_{\hbar})_1}$, and we conclude that $\phi$ cannot be a mystic reflection.

\textit{Unimodular 3}. $\Omega_3 = 2x_1x_2x_3$. 

Applying Lemma 4.3.1, it is sufficient to compute the trace series $\text{Tr}_{P_{\hbar}}(\phi, t)$ of $\phi$ of the forms $\phi_1\big\vert_{(P_{\hbar})_1} = 
\begin{bmatrix}0 & a & 0\\0 & 0 & b\\c & 0 & 0\end{bmatrix}, \phi_2\big\vert_{(P_{\hbar})_1} = \begin{bmatrix}0 & 0 & a\\b & 0 & 0\\0 & c & 0\end{bmatrix}$. Again, fix a degree $d \geq 0$ and take an arbitrary basis element $y_1^iy_2^jy_3^k$ from $(P_{\hbar})_d$: $i + j + k = d$ \cite[page 254]{DML}. Apply $\phi_1$ and $\phi_2$:
\[
\phi_1(y_1^iy_2^jy_3^k) = a^ib^jc^ky_2^iy_3^jy_1^k = a^ib^jc^k(\frac{1-\hbar}{1+\hbar})^{(i-j)k}y_1^ky_2^iy_3^j,
\]
\[
\phi_2(y_1^iy_2^jy_3^k) = a^ib^jc^ky_3^iy_1^jy_2^k = a^ib^jc^k(\frac{1-\hbar}{1+\hbar})^{(k-j)i}y_1^jy_2^ky_3^i.
\]
In both instances, the existence of the term $y_1^iy_2^jy_3^k$ within $\phi(y_1^iy_2^jy_3^k)$ is equivalent to the condition that $i = j = k$, and when this specified condition is satisfied, the coefficient of $y_1^iy_2^jy_3^k$ is identical to the coefficient of $y_1^iy_2^jy_3^k$ when we rearrange the terms lexicographically in a polynomial ring. Using the same reasoning as in Lemma 4.3.1, $\displaystyle{\text{Tr}_{P_{\hbar}}(\phi, t) = \frac{1}{(1-\lambda_1t)\cdots(1-\lambda_n t)}}$, where $\lambda_1, \cdots, \lambda_n$ are the eigenvalues of $\phi\big\vert_{(P_{\hbar})_1}$, and we conclude that $\phi$ cannot be a mystic reflection.

\smallskip

For \textit{Unimodular 6}, we pursue an alternative approach.

\textit{Unimodular 6}. $\Omega_6: x_1^3+x_1^2x_3+x_2^2x_3$. 

Our analysis in Section 3.2 reveals that there is no classical reflection. As established in Lemma 2.4.4, the standard deformation quantization $P_{\hbar}$ is a quantum polynomial ring. Therefore, our objective boils down to demonstrating that a graded automorphism $\phi$ of the form $\phi\big\vert_{(P_{\hbar})_1} = \begin{bmatrix}-\frac{1}{2}a & \pm\frac{\sqrt{3}}{2}a & 0\\\mp\frac{\sqrt{3}}{2}a & -\frac{1}{2}a & 0\\\frac{9}{8}a & \mp\frac{3\sqrt{3}}{8}a & a\end{bmatrix}$ cannot be a mystic reflection by a combination of Lemma 4.3.1 and \cite[Theorem 3.1]{KKZ}. By calculation,
\[
\begin{bmatrix}-\frac{1}{2}a & \pm\frac{\sqrt{3}}{2}a & 0\\\mp\frac{\sqrt{3}}{2}a & -\frac{1}{2}a & 0\\\frac{9}{8}a & \mp\frac{3\sqrt{3}}{8}a & a\end{bmatrix}^4 = 
\begin{bmatrix}
-\frac{1}{2}a^4 & \pm\frac{\sqrt{3}}{2}a^4 & 0\\
\mp\frac{\sqrt{3}}{2}a^4 & -\frac{1}{2}a^4 & 0\\
\frac{9}{8}a^4 & \mp\frac{3\sqrt{3}}{8}a^4 & a^4
\end{bmatrix},
\]
is not the identity matrix, and, therefore, is not a mystic reflection. 

\clearpage

For \textit{Unimodular 7}, \cite[Theorem 6.2, Corollary 6.3]{KKZ} has provided a comprehensive analysis of its standard deformation quantization $P_{\hbar}$, demonstrating the absence of quasi-reflections. As $P_{\hbar}$ is a quantum polynomial ring, it is tautological to state that $\phi$ is a quasi-reflection if and only if $\phi$ is a classical reflection.
\end{proof}

It is time to record all Poisson reflections (resp. quasi-reflections) of unimodular quadratic Poisson structures on $\Bbbk[x_1, x_2, x_3]$ (resp. standard deformation quantization of unimodular quadratic Poisson structures on $\Bbbk[x_1, x_2, x_3]$) into a table for future reference.

\begin{prop}
\cite[Proposition 2.3]{Ma1} Let $P = \Bbbk[x_1, x_2, x_3]$ be one of the unimodular quadratic Poisson algebras and let $P_{\hbar}$ be its standard deformation quantization. Then $\text{PR}(P) = \text{QR}(P_{\hbar})$ contains the following automorphisms: 
\[
\renewcommand{\arraystretch}{1.25}
\begin{tabular}{ |p{1.5cm}||p{12.75cm}| } 
\hline
\centering \textbf{Case} & \centering $\textbf{PR}\boldsymbol{(P) = }\textbf{QR}\boldsymbol{(P_{\hbar})}$
\tabularnewline
\hline
\hline
\centering 1 & \centering 
$\begin{bmatrix}-1 & 0 & 0\\a & 1 & 0\\d & 0 & 1\end{bmatrix}$
\tabularnewline
\hline
\centering 2 & \centering 
$\begin{bmatrix}1 & 0 & 0\\0 & \xi & 0\\0 & d & 1\end{bmatrix}$
\tabularnewline
\hline
\centering 3 & \centering
$\begin{bmatrix}1 & 0 & 0\\0 & 1 & 0\\0 & 0 & \xi\end{bmatrix}$, 
$\begin{bmatrix}1 & 0 & 0\\0 & \xi & 0\\0 & 0 & 1\end{bmatrix}$,
$\begin{bmatrix}\xi & 0 & 0\\0 & 1 & 0\\0 & 0 & 1\end{bmatrix}$
\tabularnewline
\hline
\centering 4 & \centering
$\begin{bmatrix}0 & -1 & 0\\-1 & 0 & 0\\b & b & 1\end{bmatrix}$, 
$\begin{bmatrix}-1 & 0 & 0\\1 & 1 & 0\\b & 0 & 1\end{bmatrix}$,
$\begin{bmatrix}1 & 1 & 0\\0 & -1 & 0\\0 & c & 1\end{bmatrix}$
\tabularnewline
\hline
\centering 5 & \centering $\emptyset$
\tabularnewline
\hline
\centering 6 & \centering $\emptyset$
\tabularnewline
\hline
\centering 7 & \centering $\emptyset$
\tabularnewline
\hline
\centering 8 & \centering 
$\begin{bmatrix}
-1 & 0 & 0\\
0 & 1 & 0\\
2 & 0 &1
\end{bmatrix}$
\tabularnewline
\hline
\centering 9 & \centering 
$\emptyset$
\tabularnewline
\hline
\end{tabular}
\]
\end{prop}

\vspace{.5cm}

\section{A Variant of The Shephard-Todd-Chevalley Theorem}
In this section, we prove Theorem 0.1, a variant of the Shephard-Todd-Chevalley Theorem for the standard deformation quantizations of unimodular Poisson structures on $\Bbbk[x_1, x_2, x_3]$. 

\begin{thm}
(Shephard-Todd-Chevalley Theorem) Let $P = \Bbbk[x_1, x_2, x_3]$ be a unimodular quadratic Poisson algebra. Let $G$ be a finite subgroup of the graded Poisson automorphism group of $P$, and $G_{\hbar}$ be the corresponding finite subgroup of the graded automorphism group of $P_{\hbar}$ under the isomorphism $\text{PAut}_{\text{gr}}(P) \cong \text{Aut}_{\text{gr}}(P_{\hbar})$. The following are equivalent:
\begin{enumerate}[label = (\arabic*)]
    \item $G$ is generated by Poisson reflections.

    \item $G_{\hbar}$ is generated by quasi-reflections.

    \item $P_{\hbar}^{G_{\hbar}}$ is Artin-Schelter regular.
\end{enumerate}
\end{thm}
\begin{proof}
(1) $\Leftrightarrow$ (2) is established by Proposition 3.3. For (2) $\Leftrightarrow$ (3), we can apply a combination of [\cite[Lemma 1.10(c)]{KKZ} and \cite[Proposition 2.5]{KKZ2} to simplify the proof, reducing it to proving $P_{\hbar}^{G_{\hbar}}$ has finite global dimension for every finite $G_{\hbar} \subseteq \text{Aut}_{\text{gr}}(P_{\hbar})$ generated by quasi-reflections. In order to prove this statement, we invoke \cite[Theorem 5.3]{KKZ2} and conclude that for every finite abelian $G_{\hbar} \subseteq \text{Aut}_{\text{gr}}(P_{\hbar})$ generated by quasi-reflections, $P_{\hbar}^{G_{\hbar}}$ has finite global dimension. This immediately resolves all cases, except when $P$ is determined by the superpotential $\Omega_4 = x_1^2x_2 + x_1x_2^2$ and $G_{\hbar}$ is generated by $\bigg\{\phi_1\big\vert_{P_1} = \begin{bmatrix}0 & -1 & 0\\-1 & 0 & 0\\a & a & 1\end{bmatrix}$, $\phi_2\big\vert_{P_1} = 
\begin{bmatrix}-1 & 0 & 0\\1 & 1 & 0\\b & 0 & 1\end{bmatrix} \bigg\}$. 

To address this particular case, define a $\Bbbk$-algebra homomorphism: $\varphi: \Bbbk\langle z_1, z_2, z_3\rangle \to P_{\hbar}$, by mapping 
\[
z_1 \mapsto \frac{1}{3}(a+b)y_1 + \frac{1}{3}(a-2b)y_2 + y_3, \hspace{.2in}
z_2 \mapsto y_1^2+y_2^2+y_1y_2, \hspace{.2in}
z_3 \mapsto 2y_1^3 + 3y_1^2y_2 - 3y_1y_2^2 - 2y_2^3.
\]
Straightforward computation shows that $[z_1,z_2] = \hbar z_3$, $[z_2,z_3] = 0$, $[z_3,z_1] = -6\hbar z_2^2$ are annihilated by $\varphi$. As such, $\varphi$ passes through the following quotient:
\begin{align*}
\widetilde{\varphi}: \Bbbk\langle z_1, z_2, z_3 \rangle/([z_1,z_2]=\hbar z_3, [z_2,z_3]=0, [z_3,z_1]=-6\hbar z_2^2) \to P_{\hbar}.
\end{align*}

The domain has a $\Bbbk$-linear basis $\{z_1^iz_2^jz_3^k: i, j, k \geq 0\}$. In the following argument, we will assume every polynomial in the domain is written with respect to this $\Bbbk$-linear basis. Suppose that $f(z_1,z_2,z_3) \in \ker\widetilde{\varphi}$. Decompose $f$ into two parts: a part consisting of all the terms containing $z_1$, called $g$, and another part consisting of all the remaining terms, called $h$. Because $f \in \ker\widetilde{\varphi}$, we have $\widetilde{\varphi}(g) = -\widetilde{\varphi}(h)$. Rewrite $g = z_1g'$ for some $g' \in \Bbbk\langle z_1, z_2, z_3\rangle$. In the equation $\widetilde{\varphi}(g) = -\widetilde{\varphi}(h)$,
\[
\text{LHS} = \frac{1}{3}(a+b)y_1\widetilde{\varphi}(g') + \frac{1}{3}(a-2b)y_2\widetilde{\varphi}(g') + y_3\widetilde{\varphi}(g').
\]
In particular, LHS must contain a term containing $y_3$, as flipping any variable does not produce any additional $y_3$ in $P_{\hbar}$. However, RHS does not contain any $y_3$, as $\varphi(z_2), \varphi(z_3) \in \Bbbk[y_1, y_2] \subseteq P_{\hbar}$. This is a contradiction unless $g = 0$. 

Now, the domain of $\widetilde{\varphi}$ can be realized as an Ore-extension $\Bbbk[z_2,z_3][z_1, \text{id}, \delta]$, where $\delta$ is the derivation satisfying $\delta(z_2) = \hbar z_3$ and $\delta(z_3) = 6\hbar z_2^2$. Since $\ker{\widetilde{\varphi}} \subseteq \Bbbk[z_2, z_3] \subseteq \Bbbk[z_2,z_3][z_1, \text{id}, \delta]$, we may restrict the domain to $\Bbbk[z_2,z_3]$ and subsequently restrict the codomain to $\Bbbk[y_1, y_2]$:
\[
\underbrace{\Bbbk[z_2,z_3] \xrightarrow{\hspace{0.25cm}\phi\hspace{0.25cm}} \Bbbk[y_1,y_2] \subseteq P_{\hbar}.}_{\widetilde{\varphi}\big\vert_{\Bbbk[z_2,z_3]}}
\]

The $\Bbbk$-algebra homomorphism $\widetilde{\varphi}$ if and only if the $\Bbbk$-algebra homomorphism $\phi$ is injective, and if and only if $\phi(z_2)$, $\phi(z_3)$ are algebraically independent. To assert algebraic independence, it is sufficient to compute the following determinant as a result of \cite[Theorem 2.4]{ER}:
\begin{align*}
\renewcommand{\arraystretch}{1.75}
\det
\begin{bmatrix}
\displaystyle{\frac{\partial \phi(z_2)}{\partial y_1}} & 
\displaystyle{\frac{\partial \phi(z_2)}{\partial y_2}}\\
\displaystyle{\frac{\partial \phi(z_3)}{\partial y_1}} & 
\displaystyle{\frac{\partial \phi(z_3)}{\partial y_2}}
\end{bmatrix} =
\begin{vmatrix}
2y_1+y_2 & y_1+2y_2\\
6y_1^2+6y_1y_2-3y_2^2 & 3y_1^2-6y_1y_2-6y_2^2
\end{vmatrix}
= 27(y_1^2y_2+y_1y_2^2)
\neq 0.
\end{align*}
Consequently, $\widetilde{\varphi}$ is injective as desired. In addition, it should be noted that the elements $z_1, z_2, z_3$ along with their commutator relations, remain invariant under the action of the group $G_{\hbar}$. The detailed calculation required to confirm this assertion is routine and, for the sake of brevity, will be omitted. This, in conjunction with the preceding argument, demonstrates that $\text{Dom}(\widetilde{\varphi})$ can be embedded into $P_{\hbar}^{G_{\hbar}}$. To establish an isomorphism, we complete the embedding into a short exact sequence:
\[
0 \to \text{Dom}(\widetilde{\varphi}) \to P_{\hbar}^{G_{\hbar}} \to P_{\hbar}^{G_{\hbar}}/\text{Dom}(\widetilde{\varphi}) \to 0. 
\]
First, since $\text{Dom}(\widetilde{\varphi})$ has a $\Bbbk$-linear basis $\{z_1^iz_2^jz_3^k: i, j, k \geq 0\}$, its Hilbert series 
\[
h_{\text{Dom}(\widetilde{\varphi})}(t) = \displaystyle{\frac{1}{(1-t)(1-t^2)(1-t^3)}}.
\]
On a different note, we can calculate the Hilbert series of $P_{\hbar}^{G_{\hbar}}$ using the Molien's Theorem:
\begin{align*}
h_{{P_{\hbar}}^{G_{\hbar}}}(t) =& \frac{1}{6}\bigg(\frac{1}{(1-t)^3} + \frac{3}{(1-t)^2(1+t)} + \frac{2}{(1-t)(1+t+t^2)}\bigg)\\
=& \frac{1}{(1-t)(1-t^2)(1-t^3)}.
\end{align*}
By the additivity of Hilbert series in a short exact sequence, the embedding $\text{Dom}\widetilde{\varphi} \hookrightarrow P_{\hbar}^{G_{\hbar}}$ translates into $\text{Dom}\widetilde{\varphi} \cong P_{\hbar}^{G_{\hbar}}$. From the isomorphism, we can realize $P_{\hbar}^{G_{\hbar}}$ as an iterated Ore-extension of an Artin-Schelter regular algebra, and therefore is Artin-Schelter regular.
\end{proof}

\

\section{Commutativity of Invariants and Deformation Quantization}

In this section, we prove taking invariant subalgebras under finite reflection groups and taking the standard deformation quantizations are two commutative processes for unimodular Poisson structures on $\Bbbk[x_1, x_2, x_3]$.

Let $P = \Bbbk[x_1, x_2, x_3]$ be a unimodular quadratic Poisson algebra. Let $G$ be a finite subgroup of the graded Poisson automorphism group of $P$ and let $G_{\hbar}$ be the corresponding finite subgroup of the graded automorphism group of $P_{\hbar}$ under the isomorphism $\text{PAut}_{\text{gr}}(P) \cong \text{Aut}_{\text{gr}}(P_{\hbar})$. After a thorough examination of the proof of \cite[Theorem 3.1]{Ma1}, there are eight possible combinations of Poisson algebras $P$ and finite Poisson reflection groups $G$. 
\begin{center}
\renewcommand{\arraystretch}{1.25}
\begin{tabular}{ |p{2.0cm}||p{3.0cm}|p{6cm}|} 
\hline
\centering \textbf{Case} &
\centering \textbf{Superpotential} &
\centering $\boldsymbol{G_{\hbar}}$ \tabularnewline
\hline
\hline
\centering 1 &
\centering $\Omega_1$ & 
\centering $\mathbb{Z}_2$
\tabularnewline
\hline
\centering 2 &
\centering $\Omega_2$ & 
\centering $\mathbb{Z}_{n_1} \times \cdots \times \mathbb{Z}_{n_m}$, $n_i | n_{i+1}$
\tabularnewline
\hline
\centering &
& \hspace{1.35cm} $(\mathbb{Z}_{l_1} \times \cdots \times \mathbb{Z}_{l_{d_1}})\times$\tabularnewline
\centering 3 & \centering $\Omega_3$
& \hspace{1.35cm} $(\mathbb{Z}_{m_1} \times \cdots \times \mathbb{Z}_{m_{d_2}})\times$ \tabularnewline
& & \hspace{1.35cm} $(\mathbb{Z}_{n_1} \times \cdots \times \mathbb{Z}_{n_{d_3}})$
\tabularnewline
& & \centering $l_i | l_{i+1}$, $m_i | m_{i+1}$, $n_i | n_{i+1}$\tabularnewline
\hline
\centering 4, 5, 6 & 
\centering $\Omega_4$ & 
\centering $\mathbb{Z}_2$
\tabularnewline
\hline
\centering 7 &
\centering $\Omega_4$ & 
\centering $S_3$
\tabularnewline
\hline
\centering 8 &
\centering $\Omega_8$ & 
\centering $\mathbb{Z}_2$
\tabularnewline
\hline
\end{tabular}
\end{center}

\smallskip

Consider the following polynomials and commutator relations in the standard deformation quantization $P_{\hbar}$:

\[
\renewcommand{\arraystretch}{1.25}
\begin{tabular}{|p{1cm}||p{4.5cm}|p{4.5cm}|p{4.5cm}| } 
\hline
\centering \textbf{Case} &
\centering $\boldsymbol{w_1}$ &
\centering $\boldsymbol{w_2}$ & 
\centering $\boldsymbol{w_3}$
\tabularnewline
\hline
\hline
\centering 1 & 
\centering $y_1^2$ & 
\centering $\frac{a}{2}y_1+y_2$ & 
\centering $\frac{d}{2}y_1+y_3$
\tabularnewline
\hline
\centering 2 & 
\centering $y_1$ & 
\centering $d_1y_2 + (1-\xi_{n_1})y_3$ & 
\centering $y_2^{l}$
\tabularnewline
\hline
\centering 3 & 
\centering $y_1^m$ & 
\centering $y_2^n$ & 
\centering $y_3^l$
\tabularnewline
\hline
\centering 4 & 
\centering $y_1y_2$ & 
\centering $-y_1 + y_2$ & 
\centering $ay_1 + y_3$
\tabularnewline
\hline
\centering 5 & 
\centering $y_1^2$ & 
\centering $\frac{b}{2}y_1+y_3$ & 
\centering $\frac{1}{2}y_1+y_2$
\tabularnewline
\hline
\centering 6 & 
\centering $y_2^2$ & 
\centering $-cy_1+y_3$ & 
\centering $2y_1+y_2$
\tabularnewline
\hline
\centering 7 & 
\centering $\frac{1}{3}(a+b)y_1 + \frac{1}{3}(2a-b)y_2+y_3$ & 
\centering $y_1^2+y_2^2+y_1y_2$ & 
\centering $2y_1^3+3y_1^2y_2-3y_1y_2^2-2y_2^3$
\tabularnewline
\hline
\centering 8 & 
\centering $y_2$ & 
\centering $y_1 + y_3$ & 
\centering $y_1^2$
\tabularnewline
\hline
\end{tabular}
\]
\[
\renewcommand{\arraystretch}{1.25}
\begin{tabular}{ |p{1cm}||p{4.5cm}|p{4.5cm}|p{4.5cm}| } 
\hline
\centering \textbf{Case} &
\centering $\boldsymbol{[w_1, w_2]}$ &
\centering $\boldsymbol{[w_2, w_3]}$ & 
\centering $\boldsymbol{[w_3, w_1]}$
\tabularnewline
\hline
\hline
\centering 1 & 
\centering $0$ & 
\centering $3\hbar w_1$ & 
\centering $0$
\tabularnewline
\hline
\centering 2 & 
\centering $(\xi_{n_1} - 1)\hbar w_1^2$ & 
\centering $2l(\xi_{n_1} - 1)\hbar w_1w_3$ & 
\centering $0$
\tabularnewline
\hline
\centering 3 & 
\centering $\hbar \frac{2mn + \textit{higher degree terms}}{(1+\hbar)^{mn}}w_1w_2$ & 
\centering $\hbar \frac{2nl + \textit{higher degree terms}}{(1+\hbar)^{mn}}w_2w_3$ & 
\centering $\hbar \frac{2ml + \textit{higher degree terms}}{(1-\hbar)^{mn}}w_1w_3$
\tabularnewline
\hline
\centering 4 & 
\centering $0$ & 
\centering $\hbar(6w_1 + w_2^2)$ & 
\centering $\hbar w_1w_2$
\tabularnewline
\hline
\centering 5 & 
\centering $-4\hbar w_1w_3$ & 
\centering $\hbar(\frac{3}{4}w_1 - w_3^2)$ & 
\centering $0$
\tabularnewline
\hline
\centering 6 & 
\centering $2\hbar w_1w_3$ & 
\centering $-\frac{3}{2}w_1 + \frac{1}{2}w_3^2$ & 
\centering $0$
\tabularnewline
\hline
\centering 7 & 
\centering $\hbar w_3$ & 
\centering $0$ & 
\centering $- 6\hbar w_2^2$
\tabularnewline
\hline
\centering 8 & 
\centering $\hbar(\frac{2}{2+\hbar}w_1w_2 + \frac{6}{2+\hbar}w_3)$ & 
\centering $\hbar(\frac{8}{(2+\hbar)^2}w_2w_3)$ & 
\centering $\hbar(\frac{8}{(2-\hbar)^2}w_1w_3)$
\tabularnewline
\hline
\end{tabular}
\]

\smallskip

\begin{lem}
The polynomials $w_i$ are invariant under the action of $G_{\hbar}$, and the commutator relations $[w_i, w_j]$ are preserved in $P_{\hbar}$. 
\end{lem}

We shall introduce one notation before presenting the proof for Lemma 5.6.1. In this section, for a pair of variables $v = (v_1, v_2, v_3), w = (w_1, w_2, w_3)$, we will use $\varphi_{vw}$ to denote the $\Bbbk$-vector space map:
\begin{align*}
\varphi_{vw}: \bigoplus_{d \geq 0}\Bbbk\left\{v_1^iv_2^jv_3^k: i + j + k = d\right\} \to& \bigoplus_{d \geq 0}\Bbbk \left\{w_1^iw_2^jw_3^k: i + j + k = d\right\}\\
v_1^iv_2^jv_3^k \mapsto& w_1^iw_2^jw_3^k.
\end{align*}

\begin{proof}
Choose an arbitrary graded Poisson automorphism $\phi \in G$. In each case, the polynomial $w_i$ is either linear, or it satisfies $w_i = y_j^{n}$ and $\phi(x_j) = \lambda x_k$, or it is contained in some $\Bbbk\langle \Bbbk y_j \oplus \Bbbk y_k \rangle$ such that $[y_j, y_k] = 0$ in the standard deformation quantization $P_{\hbar}$. As a consequence of the construction of $\widetilde{\phi}$, the polynomial $\varphi_{xy}(w_i)$ is invariant under the action of $\widetilde{\phi}$.

Regarding the commutator relations, it is challenging to prove that these relations are preserved in $P_{\hbar}$ systematically because we are working with a variant of the classic deformation quantization, and there is no natural $\Bbbk$-algebra homormorphism from $P_{\hbar}$ to $P$. Alternatively, we can conduct a manual verification of the preservation of each relation in $P_{\hbar}$ since there are a limited number of cases. For the sake of conciseness, we will omit the calculations and assumes the statement for free.
\end{proof}

\

Based on Lemma 5.1, for each case, we can define a $\Bbbk$-algebra homomorphism:
\begin{align*}
\Phi: \Bbbk\langle z_1, z_2, z_3\rangle/([z_i, z_j] = \varphi_{wz}([w_i,w_j]))_{1 \leq i, j \leq 3} \to& P_{\hbar}^{G_{\hbar}} \subseteq P_{\hbar},\\
z_1, z_2, z_3 \mapsto& w_1, w_2, w_3.
\end{align*}

In the subsequent three Lemmas, we will establish either injectivity or surjectivity for $\Phi$ in all cases.

\begin{lem}
For Case 1, Case 2, Case 4, Case 5, Case 6, Case 7, the $\Bbbk$-algebra homomorphism $\Phi$ is injective. 
\end{lem}
\begin{proof}
A common characteristic shared by Case 1, Case 2, Case 4, Case 5, Case 6, and Case 7 is the existence of a pair of indices $(i, j) \in \{1,2,3\}^{\oplus 2}$ satisfying:
\begin{enumerate}[label = (\arabic*)]
    \item the variable $y_i$ is contained in $w_j$,

    \item the variable $y_i$ is not contained in $w_k$ for all $k \in \{1,2,3\} \backslash \{j\}$,

    \item the variables $y_k, y_l$ commute for all $k, l \in \{1,2,3\} \backslash \{i\}$,

    \item the commutator bracket $[y_i, y_k]$ does not contain a term containing $y_i$ for all $k \in \{1,2,3\}$, 

    \item the variables $w_k, w_l$ commute for all $k, l \in \{1,2,3\} \backslash \{j\}$. 
\end{enumerate}

Suppose that $f \in \ker\Phi$. Decompose $f$ into two parts: a part consisting of all the terms containing $z_j$, called $g$, and another part consisting of all the remaining terms, called $h$. Because $f \in \ker{\Phi}$, we have ${\Phi}(g) = -{\Phi}(h)$. In the equation,
\[
\text{LHS} = y_i g' + \textit{remaining terms}, 
\]
for some polynomial $g'$. In particular, LHS contain at least one term containing $y_i$, as flipping any variable does not produce an additional $y_i$. However, RHS cannot contain any term containing $y_i$, as $y_i$ is not contained in $w_k$ for all $k \in \{1,2,3\} \backslash \{j\}$. This is a contradiction unless $g = 0$. 

Let $\{a, b\} = \{1,2,3\} \backslash \{j\}$ and $\{c, d\} = \{1, 2, 3\} \backslash \{i\}$. The domain of ${\Phi}$ can be realized as an Ore-extension $\Bbbk[z_a,z_b][z_j, \text{id}, \delta]$, for some appropriate choice of $\delta$. Since $\ker{\Phi} \subseteq \Bbbk[z_a, z_b] \subseteq \Bbbk[z_a,z_b][z_j, \text{id}, \delta]$, we may restrict the domain to $\Bbbk[z_a,z_b]$ and subsequently restrict the codomain to $\Bbbk[y_c, y_d]$:
\[
\underbrace{\Bbbk[z_a,z_b] \xrightarrow{\hspace{0.25cm}\Phi'\hspace{0.25cm}} \Bbbk[y_c,y_d] \subseteq P_{\hbar}.}_{\Phi\big\vert_{\Bbbk[z_a,z_b]}}
\]

The $\Bbbk$-algebra homomorphism $\Phi$ if and only if the $\Bbbk$-algebra homomorphism $\Phi'$ is injective, and if and only if $\Phi'(z_a)$, $\Phi'(z_b)$ are algebraically independent. To assert algebraic independence, it is sufficient to compute the following determinant as a result of \cite[Theorem 2.4]{ER}:
\begin{align*}
\renewcommand{\arraystretch}{1.75}
\det
\begin{bmatrix}
\displaystyle{\frac{\partial \Phi'(z_a)}{\partial y_c}} & 
\displaystyle{\frac{\partial \Phi'(z_a)}{\partial y_d}}\\
\displaystyle{\frac{\partial \Phi'(z_b)}{\partial y_c}} & 
\displaystyle{\frac{\partial \Phi'(z_b)}{\partial y_d}}
\end{bmatrix},
\end{align*}
which, in the aforementioned cases, does not vanish. Consequently, $\Phi$ is injective as desired.
\end{proof}

\smallskip

\begin{lem}
For Case 3, the $\Bbbk$-algebra homomorphism $\Phi$ is injective.
\end{lem}
\begin{proof}
This lemma becomes immediate if we repeat the former half of the proof of Lemma 5.2 (LHS, RHS argument) three times, because in this case, $z_1 \mapsto y_1^m$, $z_2 \mapsto y_2^n$, and $z_3 \mapsto y_3^l$.
\end{proof}

\smallskip

\begin{lem}
For Case 8, the $\Bbbk$-algebra homomorphism $\Phi$ is surjective.
\end{lem}
\begin{proof}
According to Theorem 4.1 and \cite[Proposition 1.1]{S}, the invariant subalgebra $P_{\hbar}^{G_{\hbar}}$ has either two generators subjecting to two homogeneous relations or three generators subjecting to three homogeneous relations. We fall into the latter category because the elements $w_1$, $w_2$, $w_3$ cannot be minimally generated by two elements in $P_{\hbar}$, and thus, not in the invariant subalgebra $P_{\hbar}^{G_{\hbar}}$ either. Furthermore, the elements $w_1$, $w_2$, $w_3$ form a minimal set of generators of $P_{\hbar}^{G_{\hbar}}$ because:
\begin{itemize}
    \item the elements $w_1$, $w_2$, $w_3$ are contained in $P_{\hbar}^{G_{\hbar}}$,

    \item the linear polynomials $w_1$, $w_2$ are linearly independent,

    \item no linear elements apart from linear combinations of $w_1$ and $w_2$ is contained in $P_{\hbar}^{G_{\hbar}}$,

    \item the quadratic polynomial $w_3$ cannot be generated by $w_1$ and $w_2$.
\end{itemize}

Given that $w_1, w_2, w_3$ lie in the image of $\Phi$, the $\Bbbk$-algebra homomorphism $\Phi$ is surjective.
\end{proof}

\smallskip

We now turn to the proof of Theorem 0.2:

\begin{thm}
Let $P = \Bbbk[x_1, x_2, x_3]$ be a unimodular quadratic Poisson algebra. Let $G$ be a finite subgroup of the graded Poisson automorphism group of $P$ and let $G_{\hbar}$ be the corresponding finite subgroup of the graded automorphism group of $P_{\hbar}$ under the isomorphism $\text{PAut}_{\text{gr}}(P) \cong \text{Aut}_{\text{gr}}(P_{\hbar})$. Define $Q_{\hbar} := P_{\hbar}^{G_{\hbar}}$, with $\hbar$ viewed as a formal parameter (as opposed to a scalar value). Then 
\begin{enumerate}[label = (\arabic*)]
    \item $Q_{\hbar}/(\hbar) \cong P^{G}$ as $\Bbbk$-algebras.

    \item For all $f, g \in Q_{\hbar}$, $fg - gf = \hbar \pi_1(f,g)$ for some $\pi_1(f,g) \in Q_{\hbar}$.

    \item $Q_{\hbar}/(\hbar)$ together with the following Poisson bracket:
    \[
    \{\overline{f}, \overline{g}\} = \overline{\pi_1(f,g)},
    \]
    where $\overline{(\hspace{.1in})}$ denotes the image under the natural projection $Q_{\hbar} \to Q_{\hbar}/(\hbar)$, is isomorphic to $P^G$ as Poisson algebras. 
\end{enumerate}
\[
\begin{tikzcd}
\rotatebox{0}{$P$} \arrow{rrrrr}{\rotatebox{0}{deformation quantization}} \arrow[swap]{ddd}{\hspace*{-0.5cm}\rotatebox{90}{invariant}} & & & & & \rotatebox{0}{$P_{\hbar}$}
\arrow{ddd}{\hspace*{0.1cm}\rotatebox{270}{invariant}}\\
\\
\\
\rotatebox{0}{$P^{G}$} & & & & & \rotatebox{0}{$P_{\hbar}^{G_{\hbar}}$} 
\arrow{lllll}{\hspace*{0.2cm}\rotatebox{0}{semiclassical limit}}
\end{tikzcd}
\]
\end{thm}
\begin{proof}
By Lemma 5.2, Lemma 5.3, Lemma 5.4, the $\Bbbk$-algebra homomorphism $\Phi$ is either injective or surjective. In either case, complete $\Phi$ to a short exact sequence by adding a cokernel (injective) or a kernel (surjective). Assign a degree to $z_i$ such that $\deg(z_i) = \deg(w_i)$, for $1 \leq i \leq 3$. It is then straightforward to observe that the commutator relations are homogeneous and $\Phi$

the $\Bbbk$-algebra homomorphism $\Phi$ is graded. For Case 1 - Case 8, the domain of $\Phi$ admits a $\Bbbk$-linear basis $\{z_1^iz_2^jz_3^k: i, j, k \geq 0\}$, leading to a Hilbert series of $\displaystyle{\frac{1}{(1-t^{\deg(w_1)})(1-t^{\deg(w_2)})(1-t^{\deg(w_3)})}}$. By applying Molien's Theorem, it can be deduced that the Hilbert series of the codomain of $\Phi$, namely $P_{\hbar}^{G_{\hbar}}$, is also given by $\displaystyle{\frac{1}{(1-t^{\deg(w_1)})(1-t^{\deg(w_2)})(1-t^{\deg(w_3)})}}$. By addivity of the Hilbert series, $\Phi$ is a $\Bbbk$-algebra isomorphism. 

Let $Q_{\hbar} := P_{\hbar}^{G_{\hbar}}$, with $\hbar$ viewed as a formal parameter instead of a scalar value. One can scrutinize the conditions outlined in Theorem 5.1. It is clear that (1) and (2) hold. For (3),
\[
\renewcommand{\arraystretch}{1.25}
\begin{tabular}{ |p{1cm}||p{4.5cm}|p{4.5cm}|p{4.5cm}| } 
\hline
\centering \textbf{Case} &
\centering $\boldsymbol{[w_1, w_2]}$ &
\centering $\boldsymbol{[w_2, w_3]}$ & 
\centering $\boldsymbol{[w_3, w_1]}$
\tabularnewline
\hline
\hline
\centering 1 & 
\centering $0$ & 
\centering $3w_1$ & 
\centering $0$
\tabularnewline
\hline
\centering 2 & 
\centering $\frac{\xi_{n_1}-1}{a_1}z_1^2$ & 
\centering $\frac{2l(\xi_{n_1}-1)}{a_1}z_1z_3$ & 
\centering 0 \tabularnewline
\hline
\centering 3 & 
\centering $2mnz_1z_2$ & 
\centering $2nlz_2z_3$ & 
\centering $2mlz_1z_3$ \tabularnewline
\hline
\centering 4 & 
\centering $-2z_1^2 - 12z_3$ & 
\centering $-2z_1z_3$ & 
\centering $0$ \tabularnewline
\hline
\centering 5 & 
\centering $z_1^2-3z_3$ & 
\centering $4z_1z_3$ & 
\centering $0$ \tabularnewline
\hline
\centering 6 & 
\centering $-z_1^2+3z_3$ & 
\centering $-4z_1z_3$ & 
\centering $0$ \tabularnewline
\hline
\centering 7 & 
\centering $z_3$ & 
\centering $0$ & 
\centering $-6z_2^2$ \tabularnewline
\hline
\centering 8 & 
\centering $z_1z_2 + 3z_3$ & 
\centering $2z_2z_3$ & 
\centering $2z_1z_3$ \tabularnewline
\hline
\end{tabular}
\]

Comparing with \cite[Theorem 3.1]{Ma1}, the Poisson algebra $Q_{\hbar}$ is isomorphic to $P^{G}$ as Poisson algebras, after suitable re-labelings of $z_1$, $z_2$, $z_3$. 
\end{proof}

\

\section{Future Work}
In this section, we reflect on some interesting insights that emerged in the proofs and propose several possible problems for future research.

\vspace{.025in}

\begin{quest}
Let $P = \Bbbk[x_1, x_2, x_3]$ be a non-unimodular quadratic Poisson algebra and let $P_{\hbar}$ be its standard deformation quantization. Is Proposition 3.3 necessarily true? To be more explicit, does $\text{QR}(P_{\hbar})$ necessarily contain no mystic reflection? Our proof for Proposition 3.3 is strongly associated with the fact that most unimodular Poisson structures satisfies $\text{Tr}_{P_{\hbar}}(\phi,t) = \displaystyle{\frac{1}{\det\left(I_3-t(\phi\big\vert_{(P_{\hbar})_1})\right)}}$ for all $\phi \in \text{Aut}_{\text{gr}}(P_{\hbar})$. However, there appears to be no rationale to substantiate such equality for a generic Poisson structure.
\end{quest}

\vspace{.025in}

\begin{quest}
Let $P = \Bbbk[x_1, x_2, x_3]$ be a non-unimodular quadratic Poisson algebra and let $P_{\hbar}$ be its standard deformation quantization. As a continuation of Question 6.1, do Theorem 4.1 and Theorem 5.1 still hold? 
\end{quest}

\vspace{.025in}

\begin{quest}
Let $P = \Bbbk[x_1, x_2, x_3]$ be a quadratic Poisson algebra. Throughout this paper, we have exclusively worked with the deformation quantization as presented in \cite{DML} with a minor alternation. If $P$ admits another non-isomorphic deformation quantization, do Theorem 4.1 and Theorem 5.1 still hold? If not, under what conditions on the deformation quantization do Theorem 4.1 and Theorem 5.1 hold?
\end{quest}

\vspace{.025in}

\begin{quest}
Let $P = \Bbbk[x_1, x_2, x_3]$ be a quadratic Poisson algebra. As a continuation of Question 6.3, is it possible to classify the deformation quantization of $P$? Is there a set of conditions, in particular, a set of homological conditions related to the second and third Poisson cohomology $PH^2$ and $PH^3$ that control non-isomorphic deformation quantizations of $P$?
\end{quest}

\

\

\bibliographystyle{alpha}
\bibliography{References}

\end{document}